\documentclass{mpi2015-cscpreprint}
\usepackage{amssymb}
\usepackage{amsthm,amsmath}
\usepackage{graphicx}
\usepackage{appendix}
\usepackage{multirow}
\usepackage{array}
\usepackage{amsmath,amssymb}
\usepackage{graphics,graphicx,epsfig,algorithm,algpseudocode}
\usepackage{subcaption}
\usepackage{booktabs}

\usepackage{acronym}
\usepackage{bookmark}
\usepackage{adjustbox}
\usepackage{csquotes}
\usepackage{dsfont}
\usepackage[font=small,labelfont=bf]{caption}
\usepackage{todonotes}
\usepackage{mathtools}
\usepackage{tcolorbox}
\usepackage{hyperref}
\providecommand{\sc}{\scshape}

\newtheorem{theorem}{Theorem}[section]

\newtheorem{remark}{Remark}[section]

\numberwithin{equation}{section}

\newcommand{\norm}[1]{\left\lVert #1\right\rVert}

\begin{document}

\title{Boundary control and periodic trajectories of semilinear Euler equations with application to hydrogen transport}

\author[1]{Alexander Zuyev}
\author[1]{Peter Benner}
\affil[1]{Max Planck Institute for Dynamics of Complex Technical Systems, Magdeburg, Germany
    \email{zuyev@mpi-magdeburg.mpg.de}, \orcid{0000-0002-7610-5621}
    \email{benner@mpi-magdeburg.mpg.de}, \orcid{0000-0003-3362-4103}
}



\abstract{
A mathematical model of gas flow in a pipeline controlled by the inlet pressure and the outlet mass flow flux is considered in the form of the isothermal Euler equations with an appropriate equation of state. Within the framework of boundary control systems, this model is transformed into a nonlinear abstract differential equation using an appropriate lifting operator. For this abstract equation, a representation in terms of Fourier coefficients is derived analytically. Conditions for the existence of periodic solutions to a broad class of nonlinear control systems with continuously differentiable controls are established in abstract spaces. This framework is applied to a realistic hydrogen transport model to evaluate periodic operating regimes under periodic fluctuations in supply and demand.
}

 \novelty{
\begin{itemize}
\item A transformation of a nonlinear system of isothermal Euler equations with boundary control into an abstract control system with a lifting operator is proposed.
\item The existence and uniqueness of periodic solutions, as well as an iterative scheme for their construction, are established for a general class of nonlinear control systems with periodic controls in a Hilbert space.
\item Integral equations for the Fourier coefficients in the expansion of the solution to the considered gas transportation model are explicitly derived.
\item Numerical simulations are performed for a realistic model of hydrogen transport to illustrate the applicability of the proposed method for constructing periodic solutions.
\end{itemize}}

\maketitle

\section{Introduction}
Control problems for gas network models have been extensively studied in the literature (see, e.g.,~\cite{osiadacz1996different,herty2010new,domschke2021gas,bermudez2022numerical}).
A wide range of publications in this area focuses on the control-theoretic treatment of the Euler equations on graphs representing mathematical models of gas transportation networks~(cf. \cite{fazeny2026optimal} and the references therein). Although the characterization of fundamental properties, such as controllability~\cite{gugat2023constrained} and observability~\cite{gugat2024observer}, is crucial within the control-theoretic framework, the design of controls and corresponding trajectories that satisfy prescribed boundary conditions remains a challenging problem, both from a theoretical perspective and in terms of efficient numerical implementation.
%
In particular, in realistic scenarios the objective is to implement time-periodic supplies in gas pipelines in order to account for periodic consumer demand. Such periodicity naturally arises in applications involving daily or weekly supply profiles, as well as in long-term planning based on seasonal variations. These situations can be naturally addressed within the framework of periodic control.
A review of the existing literature indicates that the existence of periodic solutions to nonlinear gas network models has not been addressed.
From a control-theoretic perspective, the construction of periodic trajectories for infinite-dimensional nonlinear control systems under discontinuous or low-regularity control inputs remains a challenging problem and requires further theoretical development.

In this work, we focus on an isothermal gas transportation model based on the Euler equations with boundary control. The resulting control system is formulated as an abstract nonlinear differential equation in a suitable Hilbert space using the theory of boundary control systems.  To the best of our knowledge, this work provides the first complete characterization of existence and uniqueness of solutions for this class of systems under arbitrarily low-regular control inputs.

The presentation is organized as follows. In Section 2, we consider the system of isothermal Euler equations with boundary control inputs and transform it, using a reference steady state and a representation of the deviations in terms of Riemann invariants, into a boundary control system in a Hilbert space. This system is then analyzed within the framework of boundary control theory by exploiting lifting operators. In Section 3, we establish a general result describing the existence and uniqueness of periodic solutions to abstract control systems with periodic inputs.
This existence result is proved via the Banach fixed-point theorem.
In Section 4, an analytical procedure is developed to derive a representation of the system in terms of a Fourier expansion with respect to the eigenvectors of the infinitesimal generator.
From a computational perspective, these theoretical results are translated into a finite-dimensional Galerkin formulation. In the construction that follows, we employ eigenfunctions associated with the corresponding differential operator, which acts on a space of complex-valued functions with two-dimensional complex structure.
 Finally, in Section 5, the proposed procedure for computing periodic trajectories is illustrated through numerical simulations of the hydrogen transportation model under periodic modulations of the inlet supply pressure and the outlet mass flux.

\section{An isothermal gas flow model}
\subsection{Semilinear Euler equations with boundary control}
We assume that the flow of a gas in a pipe of length $\ell$ and diameter $D$, operating under isothermal conditions and at relatively small velocities,
is described by the following nonlinear hyperbolic system~\cite{herty2010new}:
\begin{eqnarray}\label{ISO2_1}
&\frac{\partial\rho}{\partial t} + \frac{\partial}{\partial x}(\rho v) &= 0,\\
&\frac{\partial}{\partial t}(\rho v) + \frac{\partial p}{\partial x} &= - \frac{\lambda}{2D} \rho v |v| - g\rho h_x,\; x\in (0,\ell).\label{ISO2_2}
\end{eqnarray}
Here, $\rho=\rho(t,x)$, $v=v(t,x)$, and $p=p(t,x)$ denote the density, velocity, and pressure of the gas at position $x\in [0,\ell]$ and time $t$, respectively. System~\eqref{ISO2_1}--\eqref{ISO2_2} also contains the following physical parameters: the pipe friction coefficient $\lambda\ge 0$,
the gravitational constant $g$, and $h_x= \sin\alpha$, where  $\alpha$ is the inclination angle of the gas pipe relative to the horizontal plane.
This system of partial differential equations, also known as the ISO2 model~\cite{domschke2021gas}, is obtained from the isothermal Euler equations by assuming that the velocity is significantly less than the speed of sound (cf.~\cite{osiadacz1996different}) and is complemented by the equation of state
\begin{equation}\label{state_eq}
p= c^2\rho, \quad c^2 = \frac{z_0 RT }{M_g}>0,
\end{equation}
where $R$ is the universal gas constant, $T$ is the absolute temperature, $M_g$ is the gas molecular weight, and $z_0>0$ is the compressibility factor, which is assumed to be constant. Note that $c>0$ is the constant speed of sound in this model, and the case of an ideal gas corresponds to $z_0=1$.

By defining the mass flux $q=\rho v$ and exploiting equation~\eqref{state_eq} with $p>0$, we rewrite system~\eqref{ISO2_1}--\eqref{ISO2_2} with respect to $(p,q)$ in the form
\begin{eqnarray}\label{pqsys_p}
&\frac{\partial p}{\partial t} + c^2 \frac{\partial q}{\partial x}  &= 0,\\
&\frac{\partial q}{\partial t} + \frac{\partial p}{\partial x} &= - \frac{\lambda c^2}{2D} q \left| \frac{q}{p}\right| - \frac{g h_x}{c^2} p,\; x\in (0,\ell),\label{pqsys_q}
\end{eqnarray}
and introduce the boundary controls
\begin{equation}\label{pq_controls}
p_{in}(t) = p(t,0),\; q_{out}(t) = q(t,\ell).
\end{equation}
Thus, system~\eqref{pqsys_p}--\eqref{pq_controls} is controlled by the inlet pressure $p_{in}(t)$ and the
outlet mass flow rate per unit cross-sectional area $q_{out}(t)$.

\subsection{Steady-state solution and equations in deviations}
For given constant controls $p_{in}(t)=\bar p_{in}>0$ and $q_{out}(t)=\bar q>0$, system~\eqref{pqsys_p}--\eqref{pq_controls} admits the steady-state solution $(\bar p,\bar q)$, where $\bar q(x) \equiv \bar q$ and $\bar p(x)$ satisfies the initial value problem
\begin{eqnarray}\label{p0_ODE}
&\bar p'(x) &= - \frac{a {\bar q}^2 }{\bar p(x)}  - b \bar p(x),\;x\in(0,\ell),\\
&\bar p(0) &= \bar p_{in},\label{p0_IC}
\end{eqnarray}
where
\begin{equation}\label{ab}
a= \frac{\lambda c^2 }{2D },\; b =  \frac{g h_x}{c^2}.
\end{equation}
Here and in the sequel, the prime denotes differentiation with respect to the spatial variable~$x$.

The solution of~\eqref{p0_ODE}--\eqref{p0_IC} is given by
\begin{equation}\label{pbar}
\bar p (x) = \begin{cases}\sqrt{\frac{(a {\bar q}^2 +b {\bar p_{in}}^2 )e^{-2bx}-a {\bar q}^2}{b}}\;\text{if}\;b\neq 0,\\
\sqrt{         {\bar p_{in}}^2-2a {\bar q}^2  x          } \;\text{if}\;b = 0.
\end{cases}
\end{equation}
By analysing the solution~\eqref{pbar}, we conclude that the inequality $\bar p(x)>0$ holds for all $x\in [0,\ell]$ if and only if the following condition is satisfied:
\begin{equation}\label{pq_conditon}
\frac{\bar p_{in}}{\bar q} > c \sqrt{\lambda \varkappa},
\end{equation}
where
$$
\varkappa =\begin{cases}
 \frac{\ell}{D}\;\text{if}\;h_x = 0,\\
 \frac{c^2 (e^{2bL}-1)}{2Dg h_x}\;\text{if}\;h_x > 0,\\
 -\frac{1}{2Db} \min \{1,e^{-2bL}-1\}\;\text{if}\; h_x < 0.
\end{cases}
$$

By substituting
\begin{equation}\label{pq_trans}
\begin{aligned}
& p(t,x)= \bar p(x) + \tilde p(t,x),\;
q(t,x)= \bar q(x) + \tilde q(t,x),\\
& p_{in}(t) = \bar p_{in} + c u_1(t),\;
q_{out}(t) = \bar q + u_2(t)
\end{aligned}
\end{equation}
into control system~\eqref{pqsys_p}--\eqref{pq_controls}, we obtain the following system with respect to the deviations $(\tilde p, \tilde q)$ from the considered steady state:
\begin{eqnarray}\label{pqsys_ptilde}
&\frac{\partial \tilde p}{\partial t} + c^2 \frac{\partial \tilde q}{\partial x}  &= 0,\\
&\frac{\partial \tilde q}{\partial t} + \frac{\partial \tilde p}{\partial x} &= f(\tilde p,\tilde q,x) ,\; x\in (0,\ell),\label{pqsys_qtilde} \\
&\tilde p(t,0) = c u_1(t), & \tilde q(t,\ell) = u_2(t),\label{pqsys_controls}
\end{eqnarray}
where $u_1$ and $u_2$ are treated as controls and
\begin{equation}\label{f_pq}
f(\tilde p,\tilde q,x) = a \left ( \frac{{\bar q}^2}{\bar p(x)} - \frac{ (\tilde q + \bar q) |\tilde q + \bar q| }{\tilde p + \bar p(x)} \right) - b \tilde p.
\end{equation}
For small $(\tilde p,\tilde q)$, the function $f$ admits the following Taylor expansion:
\begin{equation}\label{f_taylor}
f= \left(\frac{a {\bar q}^2}{{\bar p}^2(x)} - b\right) \tilde p - \frac{2 a {\bar q}}{{\bar p}(x)}\tilde q
-\frac{a{\bar q}^2}{{\bar p}^3(x)}{\tilde p}^2
+ \frac{2a{\bar q}}{{\bar p}^2(x)} \tilde p \tilde q
- \frac{a}{{\bar p}(x)}{\tilde q}^2 +
{\mathcal O}\left(|{\tilde p}|^3 + |{\tilde q}|^3 \right).
\end{equation}

\subsection{Boundary control system with Riemann invariants}
For further analysis, we express the state of the control system~\eqref{pqsys_ptilde}--\eqref{pqsys_controls} in terms of new state functions $r_1(t,x)$ and  $r_2(t,x)$ as
\begin{equation}\label{pq_riemann}
\tilde p = c(r_1-r_2),\; \tilde q = r_1 + r_2.
\end{equation}
These $r_1$ and $r_2$ are Riemann invariants for the homogeneous part of~\eqref{pqsys_ptilde}--\eqref{pqsys_qtilde} with $f\equiv 0$.
As a result, the control system~\eqref{pqsys_ptilde}--\eqref{pqsys_controls} takes the form
\begin{equation}\label{ISO2_Riemann}
\frac{\partial}{\partial t} \begin{pmatrix}r_1 \\ r_2 \end{pmatrix} + \begin{pmatrix}c & 0 \\ 0 & -c \end{pmatrix} \frac{\partial}{\partial x} \begin{pmatrix}r_1 \\ r_2 \end{pmatrix} =   \begin{pmatrix}1/2 \\ 1/2\end{pmatrix} f(c(r_1-r_2),r_1 + r_2,x),\; x\in (0,\ell),
\end{equation}
\begin{equation}\label{BC_Riemann}
\begin{aligned}
(r_1-r_2)|_{x=0} &= u_1 (t),\\ (r_1+r_2)|_{x=\ell} &= u_2(t).
\end{aligned}
\end{equation}
Here, the controls $u_1(t)$ and $u_2(t)$ correspond, respectively, to the gas pressure at the inlet
$x=0$ and the mass flux at the outlet $x=\ell$,
after shifting from their steady-state values and applying a suitable scaling.

\subsection{Abstract representation of the control system}
%
We represent the control system~\eqref{ISO2_Riemann}--\eqref{BC_Riemann} in abstract form in the Hilbert space $H = L^2([0,\ell];{\mathbb R}^2)$ as
\begin{eqnarray}\label{op_ISO2}
&\dot y(t) &= {\cal A} y(t) + {\cal F}(y(t)), \; y(t) = (y_1(t),y_2(t))^\top\in H,\\
&{\cal B}y(t) &= u(t), \; u(t) = (u_1(t),u_2(t))^\top \in U\subset {\mathbb R}^2,
\label{op_ISO2_BC}
\end{eqnarray}
where the linear operators ${\cal A}:D({\cal A})\to H$, ${\cal B}: D({\cal B})\to {\mathbb R}^2$, and the nonlinear operator ${\cal F}: D({\cal F})\to H$ are defined by
$$
D({\cal A}) = D({\cal B}) =  H^1([0,\ell];{\mathbb R}^2),\; {\cal A} \begin{pmatrix}y_1 \\ y_2 \end{pmatrix} = c \begin{pmatrix}-y_1' \\ y_2' \end{pmatrix},
$$
$$
{\cal B} \begin{pmatrix}y_1 \\ y_2 \end{pmatrix} = \begin{pmatrix}y_1(0)-y_2(0) \\ y_1(\ell)+y_2(\ell) \end{pmatrix},
$$
\begin{equation}\label{F_op}
{\cal F}\begin{pmatrix}y_1 \\ y_2\end{pmatrix}(x) = \begin{pmatrix}1/2 \\ 1/2\end{pmatrix} f(c(y_1(x)-y_2(x)),y_1(x)+y_2(x),x),
\end{equation}
and $D({\cal F})$ consists of all $y\in H$ for which the right-hand side of~\eqref{F_op} belongs to $H$.

Clearly, if $(r_1(t,x),r_2(t,x))$ is a classical solution of~\eqref{ISO2_Riemann}--\eqref{BC_Riemann} with a control $u(t)=(u_1(t),u_2(t))^\top$, defined on some interval $t\in [0,T]$,
then $y(t)=(r_1(t,\cdot),r_2(t,\cdot))^\top$ satisfies~\eqref{op_ISO2}--\eqref{op_ISO2_BC}.

To handle the inhomogeneous boundary conditions in~\eqref{op_ISO2_BC}, we follow the theory of boundary control systems (cf.~\cite[Sect.~3.3]{curtain2012introduction},~\cite[Chap.~10]{tucsnak2009observation}).
For this purpose, we introduce the lifting operator with a constant profile:
\begin{equation}\label{P_op}
{\cal P}: {\mathbb R}^2\to D({\cal B}),\quad  {\cal P} \begin{pmatrix}u_1 \\ u_2 \end{pmatrix} = \frac12 \begin{pmatrix}u_1 + u_2 \\ u_2 - u_1 \end{pmatrix}.
\end{equation}
For a differentiable control $u(t)$ on $t\in [0,T]$, we represent the state of~\eqref{op_ISO2} in the form
\begin{equation}\label{y_form}
y(t) = {\cal P} u(t) + \xi(t).
\end{equation}
Since ${\cal B} {\cal P} u(t) = u(t)$ for all $t\in [0,T]$, substituting~\eqref{y_form} into~\eqref{op_ISO2}--\eqref{op_ISO2_BC}
formally yields the following differential equation for $\xi(t)$ with homogeneous boundary conditions:
\begin{eqnarray}\label{y_hom}
&\dot \xi(t) &= {\cal A}_0 \xi(t) - {\cal P} \dot u(t)+ {\cal F}({\cal P}u(t)+\xi(t)), \; t\in [0,T],
\\
&{\cal B}\xi (t) &= 0. \label{y_hom_BC}
\end{eqnarray}
Here, it is crucial that ${\cal A P} u(t)\equiv 0$ due to~\eqref{P_op}. The operator ${\cal A}_0: D({\cal A}_0)\to H$ is the restriction of $\cal A$ to $\mathrm{ker}\, {\cal B}$:
\begin{equation}\label{A0}
D({\cal A}_0) = \{y\in H^1([0,\ell];{\mathbb R}^2)\,|\, y_1(0)-y_2(0)=y_1(\ell)+y_2(\ell)=0\},\; {\cal A}_0 y = {\cal A} y.
\end{equation}

The operator ${\cal A}_0:D({\cal A}_0)\to H$ introduced above generates a $C_0$-semigroup $\{e^{t {\cal A}_0}\}_{t\ge 0}$ of bounded linear operators on $H$ by the Lumer--Phillips theorem.

Thus, for a control $u\in C^1([0,T];U)$ and initial data $\xi^0\in H$, we define the mild solution of the Cauchy problem~\eqref{y_hom}--\eqref{y_hom_BC} with $\xi (0)=\xi^0$ by the following integral equation:
\begin{equation}\label{mild_ytilde}
\xi(t) = e^{t {\cal A}_0} \xi^0 - \int_0^t e^{(t-s) {\cal A}_0} [{\cal P}\dot u(s)-{\cal F}({\cal P}u(s)+\xi(s))]\, ds.
\end{equation}
Hence, the corresponding mild solution $y(t)$ of~\eqref{op_ISO2}--\eqref{op_ISO2_BC} with the initial condition $y(0)=y^0 = \xi^0 + {\cal P} u(0)$ is
\begin{equation}\label{mild_y}
y(t) = {\cal P}u(t) + e^{t {\cal A}_0} (y^0 - {\cal P}u(0)) -  \int_0^t e^{(t-s) {\cal A}_0} [{\cal P}\dot u(s)-{\cal F}(y(s))]\, ds.
\end{equation}

In the case ${\cal F}=0$, formula~\eqref{mild_y} defines the unique classical solution $y(t)$ of the Cauchy problem~\eqref{op_ISO2}--\eqref{op_ISO2_BC} with $y(0)=y^0$ for any control $u\in C^2([0,T];U)$ and any initial data $y^0\in H$ such that $y^0 - {\cal P}u(0) \in D({\cal A}_0)$~\cite[Theorem~3.3.4]{curtain2012introduction}.

Formal integration by parts in~\eqref{mild_y} yields
\begin{equation}\label{y_mild}
y(t) = e^{t {\cal A}_0} y^0  - \int_0^t {\cal A}_0 e^{(t-s) {\cal A}_0} {\cal P}u(s)\, ds + \int_0^t e^{(t-s) {\cal A}_0} {\cal F}(y(s))\, ds.
\end{equation}
Formula~\eqref{y_mild} allows us to define mild solutions of the considered boundary control system in an extended functional space.
Namely, for the Hilbert space $H_1=D({\cal A}_0)$ with the norm equivalent to the graph norm of ${\cal A}_0$, we consider its dual $H_{-1}$ with respect to the pivot Hilbert space $H$.
Then $H_1\subset H \subset H_{-1}$ densely with continuous embeddings.
The semigroup $\{e^{t {\cal A}_0}\}_{t\ge 0}$ is naturally extended to a $C_0$-semigroup on $H_{-1}$~\cite[Chap.~2]{tucsnak2009observation};
we use the same notation $e^{t {\cal A}_0}: H_{-1}\to H_{-1}$ for this extension, and ${\cal A}_0:H\to H_{-1}$ for the extended generator.
Thus, for any $y^0\in H_{-1}$ and any $u\in L^\infty ([0,T];U)$,
we will treat the corresponding function $y(t)\in H_{-1}$  satisfying~\eqref{y_mild} as the mild solution of the boundary control problem~\eqref{op_ISO2}--\eqref{op_ISO2_BC} with initial condition $y(0)=y^0$.

\section{Periodic solutions of an abstract control system}

Although formula~\eqref{y_mild} provides an elegant description of mild solutions, its applicability relies on extension of the nonlinear operator $\mathcal F$ to the interpolation space $H_{-1}$. This issue requires an additional study and, for the sake of characterizing periodic solutions with better regularity, we restrict our analysis to system of the form~\eqref{y_hom} in $H$ with continuously differentiable controls.
To address the issue of existence and construction of periodic solutions with sufficient generality, we consider the class of abstract differential equations described by
\begin{equation}\label{abs_eq}
\dot {\xi}(t) = {\cal A}_0 \xi (t) + \phi(\xi(t),u(t)) - {\cal P} \dot u(t),\quad \xi(t)\in H,\; u(t)\in U\subset {\mathbb R}^m,
\end{equation}
where $H$ is a Hilbert space, ${\cal A}_0:D({\cal A}_0)\to H$ is a linear operator that generates a $C_0$-semigroup $\{e^{t{\cal A}_0}\}_{t\ge 0}$ of bounded linear operators on $H$, ${\cal P}:{\mathbb R}^m\to H$ is a bounded linear operator, and the nonlinearity $\phi:H\times U \to H$ is continuous and satisfies the uniform Lipschitz condition with respect to $\xi$ on some closed set $H_D\subset H$, i.e. there exists a constant $L\ge 0$ such that
\begin{equation}\label{Lip_cond}
\|\phi(\xi_1,u) - \phi(\xi_2,u)\| \le L \|\xi_1 - \xi_2\|\quad \text{for all}\;\xi_1,\xi_2\in H_D,\; u\in U.
\end{equation}
Clearly, system~\eqref{y_hom} is a particular form of~\eqref{abs_eq} with $m=2$ and $\phi(\xi,u)={\mathcal F}({\mathcal P}u + \xi)$.
For a given control $u\in C^1([0,\tau];U)$, the mild solution of~\eqref{abs_eq} with an initial data $\xi(0)=\xi_0\in H$ is defined as
\begin{equation}\label{xi_mild}
\xi(t)=e^{t{\cal A}_0} \xi_0 + \int_0^t e^{(t-s){\cal A}_0} \left[ \phi(\xi(s),u(s)) - {\cal P} \dot u(s) \right]ds,\; t\in [0,\tau].
\end{equation}
To formulate the existence result, we introduce the Banach space $X=C([0,\tau];H)$ and define the nonlinear operator ${\mathcal G}:X\to X$ as follows:
\begin{equation}\label{G_op}
{\mathcal G}: \xi(\cdot) \in X \mapsto {\mathcal G}(\xi)(t): = e^{t{\cal A}_0} \gamma(\xi) + \int_0^t  e^{(t-s){\cal A}_0} \left[ \phi(\xi(s),u(s)) - {\cal P} \dot u(s) \right]ds,
\end{equation}
where the operator $\gamma:X\to H$ is defined by
\begin{equation}\label{gamma_op}
\gamma: \xi(\cdot) \in X \mapsto \gamma(\xi): = \left(I-e^{\tau {\mathcal A}_0}\right)^{-1}
\int_0^\tau  e^{(\tau-s){\cal A}_0} \left[ \phi(\xi(s),u(s)) - {\cal P} \dot u(s) \right]ds.
\end{equation}
The above construction assumes that $\tau>0$ is chosen such that the resolvent
\begin{equation}\label{R_t}
R_\tau := \left(I-e^{\tau {\mathcal A}_0}\right)^{-1}
\end{equation}
is well-defined as a bounded linear operator from $H$ to $H$, i.e. $1\notin \sigma(e^{\tau {\mathcal A}_0})$, where $\sigma(\cdot)$ denotes the spectrum.
Under additional technical assumptions on the nonlinearity $\phi$ and the growth constants $M\ge 1$, $\omega$ of the semigroup $\{e^{t{\mathcal A}_0}\}_{t\ge 0}$,
\begin{equation}\label{growthbound}
\|e^{t{\mathcal A}_0}\| \le M e^{\omega t}\quad \text{for all} \; t\ge 0,
\end{equation}
we present the theorem on the existence and uniqueness of $\tau$-periodic solutions of~\eqref{abs_eq} below.

\begin{theorem}\label{thm1}
Let $\tau>0$ be such that $1\notin \sigma(e^{\tau {\mathcal A}_0})$,
and let $u\in C^1([0,\tau];U)$ be a control such that $u(0)=u(\tau)$.
Assume that there is a closed set $H_D\subset H$
such that the Lipshitz condition~\eqref{Lip_cond} holds and $X_D=C([0,\tau];H_D)$ is invariant for $\mathcal G$.
If
\begin{eqnarray}
\alpha:=M^* L (1+\bar M \|R_\tau\|)<1,\label{contr_const}
\end{eqnarray}
where
\begin{equation}\label{M_constants}
\bar M := \mathrm{max}\{M,Me^{\omega \tau}\},\; M^*:= \begin{cases}M\tau, & \omega=0,\\ \frac{M(e^{\omega\tau}-1)}{\omega}, & \omega\neq 0, \end{cases}
\end{equation}
then there exists a unique mild solution $\xi^*(t)\in H_D$ of~\eqref{abs_eq} satisfying $\xi^*(0)=\xi^*(\tau)$.
For any $\xi^1\in X_D$, the sequence $\xi^{n+1}:={\mathcal G}(\xi^n)$, $n=1,2,\dots$, converges to $\xi^*$ in $X_D$ as $n\to \infty$.
\end{theorem}

\begin{proof}
We check the contraction condition for the nonlinear operator $\mathcal G$. For any $\xi_1,\xi_2\in X_D$,
\begin{equation}\label{G_contr}
\begin{aligned}
\norm{{\mathcal G}(\xi_1)-{\mathcal G}(\xi_2)}_{X}
&= \sup_{t\in[0,\tau]}
\norm{
 e^{tA_0}\bigl(\gamma(\xi_1)-\gamma(\xi_2)\bigr)
 + \int_0^t e^{(t-s){\mathcal A}_0}
 \bigl(\phi(\xi_1(s),u(s))-\phi(\xi_2(s),u(s))\bigr)\,ds
}_H \\
&\le M\max\{1,e^{\omega\tau}\}
\norm{\gamma(\xi_1)-\gamma(\xi_2)}_H \\
&\quad + ML \sup_{t\in[0,\tau]}
\int_0^t e^{\omega(t-s)}
\norm{\xi_1(s)-\xi_2(s)}_H\,ds.
\end{aligned}
\end{equation}
Moreover,
\begin{equation}\label{gamma_contr}
\begin{aligned}
\norm{\gamma(\xi_1)-\gamma(\xi_2)}_H
&\le L \norm{R_\tau}
\int_0^\tau \left( \norm{e^{(\tau-s){\mathcal A}_0}}\cdot
\norm{\xi_1(s)-\xi_2(s)}_H\right)ds \\
&\le LM^*\norm{R_\tau}
\cdot
\norm{\xi_1 - \xi_2}_{X},
\end{aligned}
\end{equation}
where the norm without subscript stands for the operator norm in the space of bounded linear operators ${\mathcal L}(H,H)$.

By combining inequalities~\eqref{G_contr} and~\eqref{gamma_contr}, we conclude that the operator ${\mathcal G}:X_D\to X_D$ is contractive,
$$
\norm{{\mathcal G}(\xi_1)-{\mathcal G}(\xi_2)}_{X} \le \alpha \norm{\xi_1 - \xi_2}_{X}\quad \text{for all}\;\; \xi_1,\xi_2\in X_D.
$$
where the constant $\alpha<1$ is defined in~\eqref{contr_const}.

The assertion of Theorem~\ref{thm1} follows from the Banach fixed point theorem for the operator $\mathcal G$ defined by~\eqref{G_op} and~\eqref{gamma_op}.
\end{proof}

\begin{remark}\label{rem1}
Theorem~\ref{thm1} generalizes the simple iteration schemes, developed in~\cite{ZuyB24,BenCZ23} for finite-dimensional systems, to the case of nonlinear control systems in a Hilbert space.
\end{remark}

\section{Eigenfunction expansion of the solutions}

For further analysis, we derive a suitable coordinate representation of integral equations~\eqref{mild_ytilde}--\eqref{y_mild} by exploiting the eigenfunctions of ${\cal A}_0$
and treating the underlying Hilbert space $H$ as a complex Hilbert space.
With this objective, we consider the corresponding spectral problem:
$$
{\cal A}_0 y = \lambda y,\quad y=(y_1,y_2)^\top \in D({\cal A}_0),
$$
which is equivalent to:
$$
\begin{aligned}
y_1'(x) &= -\lambda c y_1(x),\\
y_2'(x)&= \lambda c y_2(x),\\
y_1(0)&=y_2(0),\\
y_1(\ell)&=-y_2(\ell).
\end{aligned}
$$
The eigenvalues of this spectral problem are $\lambda=\lambda_k$ with the corresponding normalized eigenfunctions $y=y^{(k)}$:
\begin{equation}\label{eigenvalues}
\lambda_k = \frac{i c  \pi (2k+1)}{2\ell},\; y^{(k)}(x) = \begin{pmatrix}y^{(k)}_1(x)  \\ y^{(k)}_2(x) \end{pmatrix}= \frac{1}{\sqrt{2\ell}}\begin{pmatrix} e^{-\lambda_k x/ c} \\ e^{\lambda_k x/ c} \end{pmatrix},\; k\in{\mathbb Z}.
\end{equation}
Since the operator ${\cal A}_0:D({\cal A}_0)\to H$ is skew-adjoint and  ${\cal A}_0^{-1}:H\to H$ is compact, the eigenfunctions $\{y^{(k)}\}_{k\in {\mathbb Z}}$ form an orthonormal basis of $H$.
This allows us to propose the following eigenvalue expansion result for the solutions of~\eqref{mild_ytilde}.

\begin{theorem}\label{thm2}
Let $\xi(t)\in H$ be a mild solution of~\eqref{mild_ytilde} with the initial condition $\xi(0)=\xi^0\in H$ and a control $u\in C^1([0,\tau];U)$. Then the Fourier coefficients $\{q_k(t)\}_{k\in \mathbb Z}$ in the expansion
\begin{equation}\label{xi_fourier}
\xi(t) = \sum_{k\in \mathbb Z} q_k(t) y^{(k)}
\end{equation}
satisfy (formally) the following infinite system of integral equations:
\begin{equation}\label{q_representation}
\begin{aligned}
q_k(t) &= e^{\lambda_k t} q_k^0  + \sum_{j=1}^2 p_{jk} \left(e^{\lambda_k t} u_j(0)-u_j(t)\right)
- \sum_{j=1}^2 p_{jk} \lambda_k \int_0^t e^{\lambda_k (t-s)}u_j(s)ds\\
&+\sum_{n\in\mathbb Z}  \psi^{(k)}_n  \int_0^t e^{\lambda_k(t-s)}\left( q_n(s)+ \sum_{j=1}^2p_{jn}u_j(s) \right) ds\\
&+ \sum_{n,m\in\mathbb Z} \psi^{(k)}_{nm}  \int_0^t e^{\lambda_k(t-s)} \left( q_n(s)+ \sum_{j=1}^2p_{jn}u_j(s) \right)
\left( q_m(s)+ \sum_{j=1}^2p_{jm}u_j(s) \right)
ds,\;\;k\in {\mathbb Z}.
\end{aligned}
\end{equation}
Here, $q_k^0 = \left<\xi(0),y^{(k)}\right>_H$ and
\begin{equation}\label{p_coeff}
p_{1k} = \frac{i\sqrt{2\ell}}{\pi (2k+1)},\; p_{2k} =\frac{(-1)^k\sqrt{2\ell}}{\pi (2k+1)},\; \beta_k= \frac{\pi(2k+1)}{2\ell},\quad k\in {\mathbb Z},
\end{equation}
 \begin{equation}\label{psi_k}
 \psi^{(k)}_n = \sqrt{\frac{2}{\ell}} \int_0^\ell \Psi^{(1)}_n(x) \cos(\beta_k x)dx,\;
  \psi^{(k)}_{nm} = \sqrt{\frac{2}{\ell}} \int_0^\ell \Psi^{(2)}_{nm}(x) \cos(\beta_k x)dx,
 \end{equation}
\begin{equation}\label{Psi_fun}
\begin{aligned}
\Psi^{(1)}_k(x) &=
\frac{ic}{\sqrt{2\ell}}\left(b-\frac{a {\bar q}^2}{{\bar p}^2(x)} \right) \sin\beta_kx - \frac{\sqrt{2}a {\bar q}}{\sqrt{\ell}{\bar p}(x)}\cos\beta_k x,\\
\Psi^{(2)}_{km}(x) &=
\frac{ac^2{\bar q}^2}{\ell{\bar p}^3(x)}\sin\beta_k x \sin\beta_m x
- \frac{2iac{\bar q}}{\ell{\bar p}^2(x)} \sin\beta_k x \cos\beta_m x
- \frac{a}{\ell{\bar p}(x)}\cos\beta_k x \cos\beta_m x.
\end{aligned}
\end{equation}
\end{theorem}
\begin{proof}
For $\xi\in H$, we denote its expansion with respect to the eigenfunction of ${\cal A}_0$ as
$$
\xi = \sum_{k\in \mathbb Z} q_k y^{(k)},\; q_k = \left<\xi,y^{(k)}\right>_H,
$$
so the infinite vector $q=(q_k)_{k\in {\mathbb Z}}\in \ell^2$ is composed of the coordinated of $\xi$ in the basis $\{y^{(k)}\}_{k\in {\mathbb Z}}$.
In this basis, the operator $e^{t{\cal A}_0}$ is diagonal with entries $e^{\lambda_k t}$, $k\in\mathbb Z$.
We also express the action of $\cal P$ in~\eqref{P_op} as
$$
{\cal P} u = {\cal P}_1 u_1 + {\cal P}_2 u_2,\; {\cal P}_1 = \begin{pmatrix}1/2 \\ -1/2\end{pmatrix}=\sum_{k\in\mathbb Z} p_{1k} y^{(k)},\; {\cal P}_2 = \begin{pmatrix}1/2 \\ 1/2\end{pmatrix}=\sum_{k\in\mathbb Z} p_{2k} y^{(k)},
$$
where $p_{1k} = \left<{\cal P}_1,y^{(k)} \right>_H$ and $p_{2k} = \left<{\cal P}_2,y^{(k)}\right>_H$ are computed in~\eqref{p_coeff}.
Then,
\begin{equation}\label{F_shifted}
{\mathcal F}({\mathcal P} u + \xi) = {\mathcal F} \left(\sum_{k\in{\mathbb Z}} \tilde q_k y^{(k)}\right),\quad
\tilde q_k := q_k+ \sum_{j=1}^2p_{jk}u_j,\; k\in\mathbb Z.
\end{equation}
Using the definition of $\mathcal F$ in~\eqref{F_op} and exploiting the equalities
\begin{equation}\label{y_sincos}
\begin{aligned}
& y^{(k)}_1(x) + y^{(k)}_2(x) = \sqrt{\frac{2}{\ell}} \cos\beta_k x,\\
& y^{(k)}_1(x) - y^{(k)}_2(x) = -i\sqrt{\frac{2}{\ell}} \sin\beta_k x,
\end{aligned}
\end{equation}
for the components of $y^{(k)}$ in~\eqref{eigenvalues}, we rewrite formula~\eqref{F_shifted} in the form:
\begin{equation}\label{F_composition}
{\mathcal F}({\mathcal P} u + \xi)(x) = \begin{pmatrix}1/2 \\ 1/2\end{pmatrix} f\left( -ic \sqrt{\frac2{\ell}}\sum_{k\in{\mathbb Z}}\tilde q_k \sin \beta_k x,
\sqrt{\frac2{\ell}}\sum_{k\in\mathbb Z}\tilde q_k \cos \beta_k x, x\right).
\end{equation}
Here, the function $f$ is defined by~\eqref{f_pq}.
By replacing the function $f$ with its Taylor expansion~\eqref{f_taylor} and truncating higher-order terms, we obtain
the following quadratic approximation for the components of ${\mathcal F}=({\mathcal F}_1,{\mathcal F}_2)^\top$ in~\eqref{F_composition} for small $\tilde q_k$:
\begin{equation}\label{F_approx}
{\mathcal F}_1({\mathcal P} u + \xi)(x) = {\mathcal F}_2({\mathcal P} u + \xi)(x)=\sum_{k\in\mathbb Z} \Psi^{(1)}_k(x) \tilde q_k+ \sum_{k,m\in\mathbb Z} \Psi^{(2)}_{km}(x) \tilde q_k \tilde q_m,
\end{equation}
where the functions $\Psi^{(1)}_k(x)$ and $\Psi^{(2)}_{km}(x)$ are defined in~\eqref{Psi_fun}.

For $u(s)\in U$ and $\xi(s)\in H$, we compute the Fourier series expansion of ${\mathcal F}({\mathcal P} u(s) + \xi(s))$:
 \begin{equation}\label{F_Fourier}
 {\mathcal F}({\mathcal P} u(s) + \xi(s)) = \sum_{k\in \mathbb Z} {\hat f}_k(s) y^{(k)},
 \end{equation}
 where
 \begin{equation}\label{f_k_expansion}
 \begin{aligned}
 {\hat f}_k(s) &=  \left<{\mathcal F}({\mathcal P} u(s) + \xi(s)),y^{(k)}\right>_H =\sqrt{\frac{2}{\ell}}\int_0^\ell {\mathcal F}_1({\mathcal P} u(s) + \xi(s)) \cos(\beta_k x)dx \\
 & = \sum_{n\in\mathbb Z} \psi^{(k)}_n \tilde q_n(s) + \sum_{n,m\in\mathbb Z} \psi^{(k)}_{nm} \tilde q_n(s) \tilde q_m(s),
 \end{aligned}
 \end{equation}
 and the coefficients $\psi^{(k)}_n$, $\psi^{(k)}_{nm}$ are given by~\eqref{psi_k}.

 Finally, for $u\in C^1([0,\tau];U)$, the substitution of $\xi(t) = \sum_{k\in\mathbb Z}q_k(t)y^{(k)}$ into~\eqref{mild_ytilde} with taking into account~\eqref{F_Fourier} and~\eqref{f_k_expansion} results in the equations~\eqref{q_representation}.
\end{proof}

\begin{remark}\label{rem1}
If the vector of Fourier coefficients $(q_k(t))_{k\in\mathbb Z}\in\ell^2$ of $\xi(t)\in H$ satisfies system~\eqref{q_representation} for all $t\in [0,\tau]$, $k\in\mathbb Z$ then, because of~\eqref{y_form}, the corresponding solution of~\eqref{op_ISO2}--\eqref{op_ISO2_BC} is represented as
\begin{equation}\label{y_fourier}
y(t)={\cal P} u(t) + \sum_{k\in\mathbb Z}q_k(t)y^{(k)},\quad t\in [0,\tau].
\end{equation}
\end{remark}

\begin{remark}\label{rem2}
As the operators $e^{t{\mathcal A}_0}$ are diagonal in the basis $\{y^{(k)}\}$, with entries $e^{\lambda_k t}$ and $\lambda_k$ defined by~\eqref{eigenvalues}, then the resolvent condition~\eqref{R_t} for this operator ${\mathcal A}_0$ is equivalent to:
\begin{equation}\label{rescond_lambda}
\inf_{k\in\mathbb Z} \left|e^{\lambda_k \tau} -1\right| >0.
\end{equation}
This condition means that the sequence $\{e^{\lambda_k \tau} (\mathrm{mod}\,2\pi i)\}_{k\in\mathbb Z}$ is separated from zero; equivalently, the sequence $\{(2k+1)\tau^* (\mathrm{mod}\,1)\}_{k\in\mathbb Z}$ is separated from zero, where
$$
\tau^* = \frac{c\tau}{4\ell}.
$$
We conclude that~\eqref{rescond_lambda} holds for every $\tau>0$ such that the corresponding $\tau^* = n/m$ is a rational number with $n,m\in\mathbb N$, $(n,m)=1$, and $m$ is even.
\end{remark}

\section{Numerical simulations}

For a reference steady-state pressure ${\bar p}_{in}>0$ and steady-state outlet mass flow rate per
unit cross-sectional area $\bar q>0$, the steady state pressure profile $\bar p(x)$ is defined by~\eqref{pbar}.
Then, for physical control inputs in the form of $p_{in}(t)$ and $q_{out}(t)$, we compute the controls $u_j(t)$ of~\eqref{op_ISO2}--\eqref{op_ISO2_BC} by~\eqref{pq_trans} as
$$
u_1(t) = \frac{p_{in}(t)-{\bar p}_{in}}{c},\; u_2(t) = q_{out}(t)-\bar q.
$$

Assuming that such controls are given as $C^1$-functions on a time interval $[0,\tau]$, we approximate the solutions $y(t)$ of the abstract differential equation~\eqref{y_form} by their truncated Fourier expansions $\hat y(t)$ from~\eqref{y_fourier} up to the coefficients $q_k(t)$ with $|k|\le K$ for some fixed $K\ge 0$:
\begin{equation}\label{y_fourier_truncated}
\hat y(t)={\cal P} u(t) + \sum_{|k|\le K}\hat q_k(t)y^{(k)}.
\end{equation}
Here, the Fourier coefficients $\hat q_k(t)$ are obtained by truncating all terms with $|k|,|m|,|n|>M$, in formula~\eqref{q_representation} from Theorem~\ref{thm2}:
\begin{equation}\label{q_truncated}
\begin{aligned}
\hat q_k(t) &= e^{\lambda_k t} {\hat q}_k^0  + \sum_{j=1}^2 p_{jk} \left(e^{\lambda_k t} u_j(0)-u_j(t)\right)
- \sum_{j=1}^2 p_{jk} \lambda_k \int_0^t e^{\lambda_k (t-s)}u_j(s)ds\\
&+\sum_{|n|\le K}  \psi^{(k)}_n  \int_0^t e^{\lambda_k(t-s)}\left( \hat q_n(s)+ \sum_{j=1}^2p_{jn}u_j(s) \right) ds\\
&+ \sum_{|n|,|m|\le K} \psi^{(k)}_{nm}  \int_0^t e^{\lambda_k(t-s)} \left( \hat q_n(s)+ \sum_{j=1}^2p_{jn}u_j(s) \right)
\left( \hat q_m(s)+ \sum_{j=1}^2p_{jm}u_j(s) \right)
ds,\\
{\hat q}_k^0 &= \left<y(0)-{\mathcal P}u(0),y^{(k)}\right>_H,\quad |k|\le K.
\end{aligned}
\end{equation}
The above system of nonlinear integral equations with respect to $(\hat q_k)_{|k|\le K}$ can be approximately solved by the Picard-type iterations,
starting with the constant initial approximation $\hat q_k(t)\equiv{\hat q}_k^0$ for all $k=-K,..., K$.
This scheme has been implemented numerically in \texttt{MATLAB R2026a}.
If $(\hat q_k(t))_{|k|\le K}$ is a solution of~\eqref{q_truncated}, then the approximation of the corresponding solution of~\eqref{pqsys_p}--\eqref{pqsys_q} by formulas~\eqref{pq_riemann} and~\eqref{y_sincos} has the following physical components:
\begin{equation}\label{pq_approximated}
\begin{aligned}
p(t,x) &= \bar p(x) - \frac{ic\sqrt{2}}{\sqrt{\ell}}\sum_{|k|\le K} \left(\hat q_k(t) + \sum_{j=1}^2 p_{jk} u_j(t)\right)\sin\beta_k x,\\
q(t,x) & = \bar q + \sqrt{\frac{2}{\ell}}\sum_{|k|\le K} \left(\hat q_k(t) + \sum_{j=1}^2 p_{jk} u_j(t)\right)\cos\beta_k x,\quad x\in [0,\ell].
\end{aligned}
\end{equation}

For the numerical simulations, we choose the following realistic values of parameters of the Euler equations~\eqref{pqsys_q}--\eqref{pqsys_p} and the equation of state~\eqref{state_eq} for the transport of hydrogen in a pipe:
\begin{equation}\label{params}
\ell = 5\cdot 10^4\,m,\; D= 0.5\,m,\; h_x = 0,\;z_0 = 1.05,\;\lambda=0.01,\; T=288.15\,K,
\end{equation}
$$
R = 8.314462618\,\frac{J}{mol\cdot K},\; M_g = 2.01588\cdot 10^{-3}\,\frac{kg}{mol},\;
c=\sqrt{\frac{z_0 RT }{M_g}}\approx 1117 \, \frac{m}{s},\; a = \frac{\lambda c^2 }{2D }\approx 12479 \;\, \frac{m}{s^2},
$$
$$
{\bar p}_{in} = 8\cdot 10^6\,Pa,\; {\bar q} = 64\,\frac{kg}{m^2 \cdot s}.
$$
The steady-state pressure profile is defined by~\eqref{pbar} as
$\bar p(x) = \sqrt{{\bar p_{in}}^2-2a {\bar q}^2  x}>0$ for all $x\in [0,\ell]$, and $\bar p(\ell) \approx 7.674\cdot 10^6 \, Pa$.

To show the applicability of the construction in Theorem~\ref{thm1}, we approximate the action of the operator $\mathcal G$,
defined as $\eta = {\mathcal G}(\xi)$ for $\xi\in X$ by formulas~\eqref{G_op}--\eqref{gamma_op}, in terms of truncated
expansions $\xi(t) \approx \sum_{|k|\le K} \hat q_k(t) y^{(k)}$ and $\eta(t) \approx \sum_{|k|\le K} \hat \eta_k(t) y^{(k)}$
as $\hat\eta = {\mathcal G}_K(\hat q)$.
Then the simple iteration method of finding the fixed points of ${\mathcal G}_K$ (i.e. the functions $\hat q = (\hat q_{-K},...,\hat q_{-K})^\top:[0,\tau]\to \mathbb C^{2K+1}$ such that ${\mathcal G}_K(\hat q)= \hat q$) is implemented numerically in $\texttt{MATLAB R2026a}$.
For all subsequent simulations, we choose $K=2$. In this case, the truncated version of condition~\eqref{rescond_lambda} holds for any time horizon $\tau$ proportional to $1\, s$: $e^{\lambda_k}\neq 1$ for all $k=-2\,...\,2$.

In the first simulation, we take $\tau= 1\, s$ and consider smooth $\tau$-periodic inputs $p_{in}(t)$ and $p_{out}(t)$ of two harmonic components
with maximal deviations from the steady-state values of 0.25\% in pressure and 1\% in flow rate, as shown in Fig.~\ref{fig:input1}(a).
The periodic pressure $p(t,x)$ and flow flux $q(t,x)$ profiles, expressed by formulas~\eqref{pq_approximated} using the vector of generalized coordinates $\hat q(t)=(\hat q_{-2}(t),\hat q_{-1}(t),\hat q_{0}(t),\hat q_{1}(t),\hat q_{2}(t),)^\top$,
are shown in Fig.~\ref{fig:input1}(a). In these computations, the relative residual in the periodic condition after 10 iterations of the proposed method is $\|\hat q(\tau)-\hat q(0)\|\approx 2.185\cdot 10^{-18}\cdot \|\hat q(0)\|$.

\begin{figure}[ht]
 \centering

    \begin{subfigure}{0.48\textwidth}
        \centering
        \includegraphics[width=\linewidth]{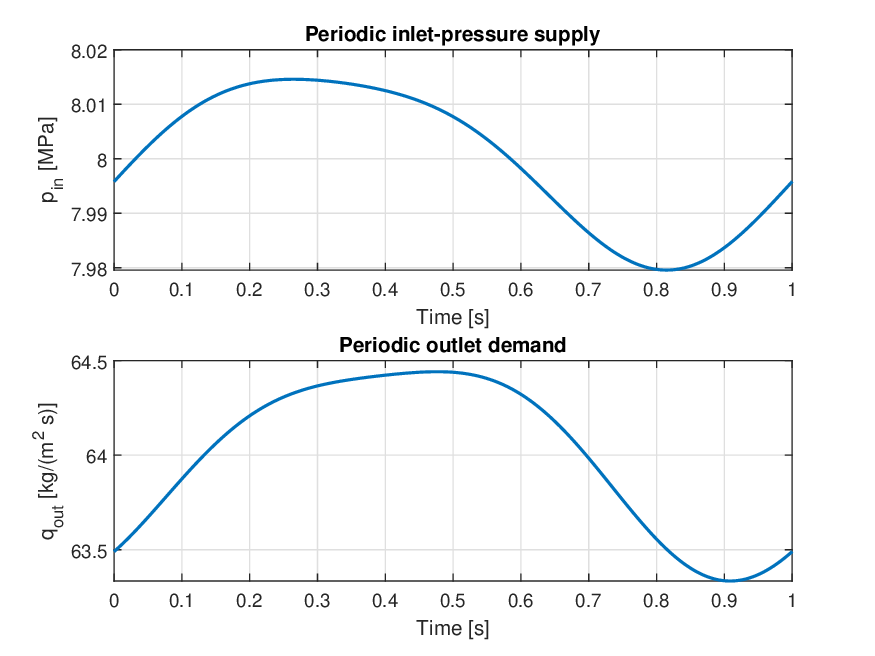}
        \caption{Pressure and flow rate controls}
        \label{fig:left}
    \end{subfigure}
    \hfill
    \begin{subfigure}{0.48\textwidth}
        \centering
        \includegraphics[width=\linewidth]{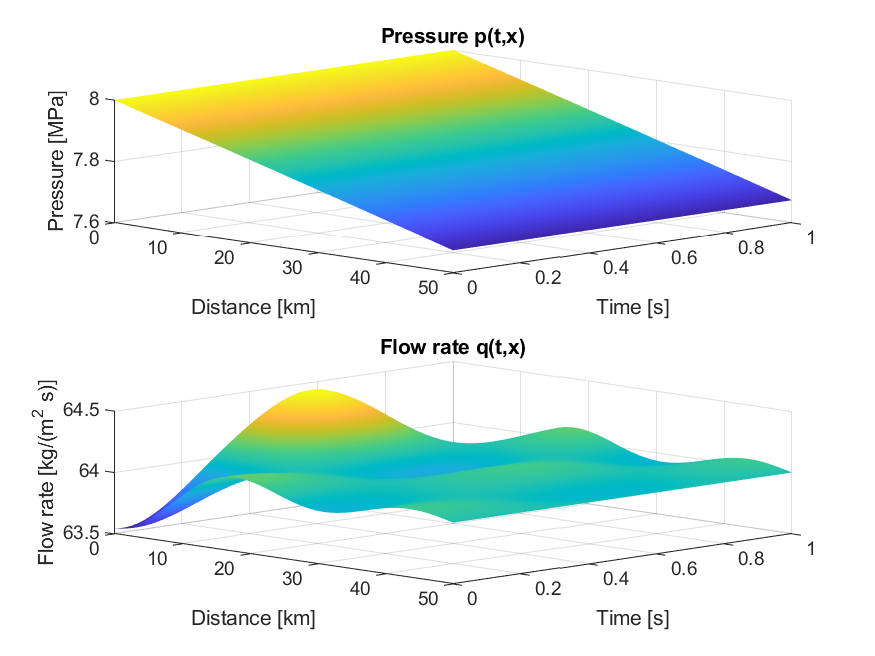}
        \caption{Profiles of $p(t,x)$ and $q(t,x)$}
        \label{fig:right}
    \end{subfigure}

    \caption{Periodic simulations for $\tau=1\,s$.}
    \label{fig:input1}
\end{figure}

To illustrate the behavior of the considered model on a larger time horizon of $\tau=24$ hours and with larger amplitudes of
7\% daily fluctuations in the inlet pressure and 25\% amplitude of fluctuations in the outlet flow rate, respectively, we present simulation results in Fig.~\ref{fig:profiles_24h}. After 41 iterations, the relative residual in the periodic boundary conditions is $\|\hat q(\tau)-\hat q(0)\|\approx 5.061\cdot 10^{-6}\cdot \|\hat q(0)\|$.

\begin{figure}[ht]
    \centering

    \begin{subfigure}{0.48\textwidth}
        \centering
        \includegraphics[width=\linewidth]{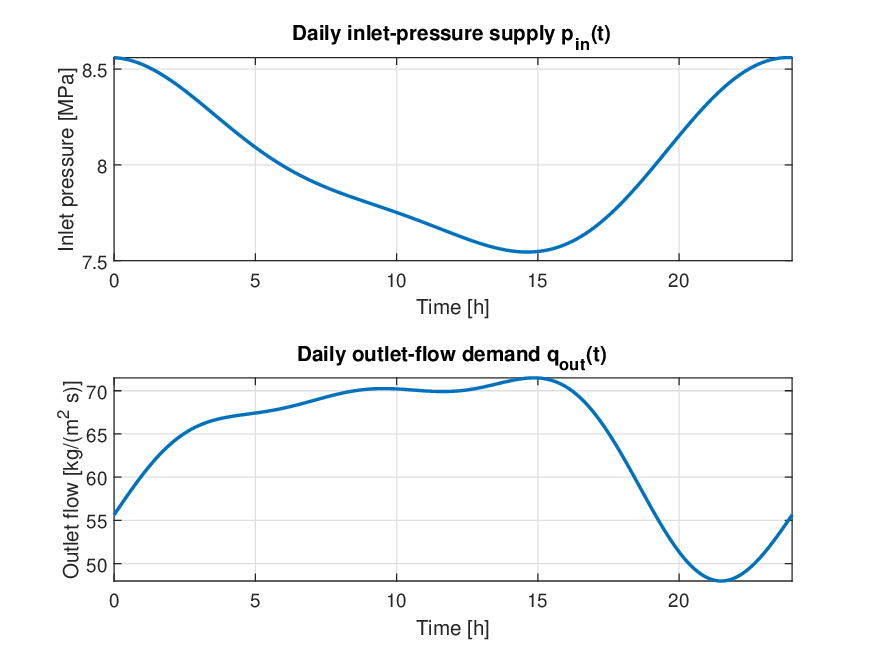}
        \caption{Pressure and flow rate controls}
        \label{fig:left}
    \end{subfigure}
    \hfill
    \begin{subfigure}{0.48\textwidth}
        \centering
        \includegraphics[width=\linewidth]{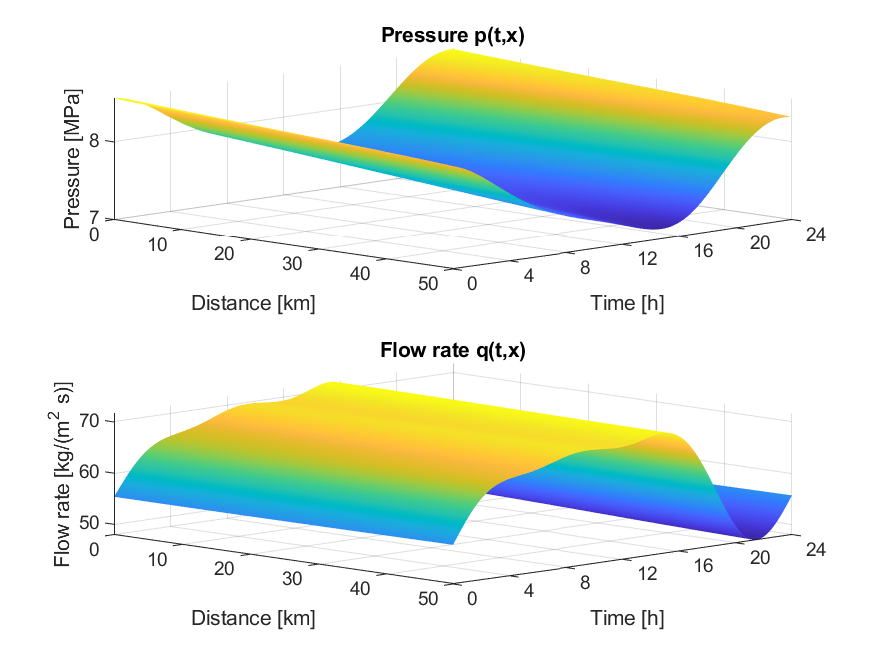}
        \caption{Profiles of $p(t,x)$ and $q(t,x)$}
        \label{fig:right}
    \end{subfigure}

    \caption{A $24$-hour periodic regime.}
    \label{fig:profiles_24h}
\end{figure}

\section{Conclusion}
Although the derivation of the approximate system for the Fourier coefficients is presented in Sections 4 and 5 for a quadratic approximation of the nonlinearity, the right-hand-side expansion in~\eqref{f_k_expansion} of Theorem~\ref{thm2} can potentially be extended to Taylor approximations of arbitrary order in~\eqref{f_taylor}. Indeed, the structure of the corresponding trigonometric series naturally admits such an extension, allowing higher-order terms to be incorporated in a systematic manner.

It should be emphasized that Theorem~\ref{thm1}, formulated in Section 3, establishes the existence and uniqueness of periodic solutions for abstract differential equations involving a general class of infinitesimal generators ${\mathcal A}_0$
 and a broad class of Lipschitz nonlinearities. As a direction for future research, we plan to extend this approach to gas network models on graphs and to investigate other classes of infinitesimal generators. For instance, when diffusion plays a crucial role in the transportation dynamics, the analysis of periodic solutions to nonlinear parabolic models remains a challenging problem.

 Additionally, the mathematical analysis of the transportation of gas mixtures, such as hydrogen and hydrocarbons, in the context of sustainable decarbonization strategies for gas networks represents another promising area of application for Theorem~\ref{thm1}.
 We plan to extend the considered $2\times 2$ hyperbolic model to models for gas mixtures~\cite{lozano2021speed,gugat2024observer} that allow the components to propagate at different speeds.
 In this context, the transformation of a boundary-controlled system into a semigroup representation with an infinitesimal generator defined under homogeneous boundary conditions constitutes the principal contribution of the present work and provides a foundation for future generalizations.
A considerable problem in this direction is the construction of lifting operators that take into account more general boundary conditions at the vertices of a gas transportation graph while admitting accurate numerical approximations.

\section*{Acknowledgments}%
The authors gratefully acknowledge the funding by the
European Regional Development Fund (ERDF) within the programme Research
and Innovation — Grant Number ZS/2023/12/182138.

\end{document}